\documentclass[11pt]{amsart}

\usepackage{amsmath,amssymb,amsthm}
\usepackage{thmtools,enumerate}
\usepackage{mathtools}
\usepackage{array}
\usepackage{mathrsfs}  

\usepackage[german,english]{babel}
\usepackage[autostyle]{csquotes}

\usepackage[all]{xy}
\usepackage{pstricks}

\usepackage{tikz-cd}
\usetikzlibrary{babel}

\usepackage{hyperref}
\hypersetup{
	colorlinks=true,
	linkcolor=red,
	citecolor=blue}

\theoremstyle{plain}

\newtheorem{thm}{Theorem}[section]
\newtheorem{prop}[thm]{Proposition}
\newtheorem{lem}[thm]{Lemma}
\newtheorem{cor}[thm]{Corollary}

\newtheorem{thmA}{Theorem}

\theoremstyle{definition}

\newtheorem{dfn}[thm]{Definition}
\newtheorem{dfn-prop}[thm]{Definition-Proposition}
\newtheorem{rem}[thm]{Remark}

\newtheorem{notation}[thm]{Notation}

\newcommand{\Z}{\mathbb{Z}}

\newcommand{\Q}{\mathbb{Q}}
\newcommand{\R}{\mathbb{R}}
\newcommand{\C}{\mathbb{C}}
\newcommand{\pr}{\mathbb{P}}

\newcommand{\cal}[1]{\mathcal{#1}}
\newcommand{\smvee}{{\scriptscriptstyle{\vee}}}
\newcommand{\smperp}{{\scriptscriptstyle{\perp}}}
\newcommand{\smone}{{\scriptscriptstyle{(1)}}}
\newcommand{\smonestar}{{\scriptscriptstyle{(1)*}}}
\newcommand{\smonestarstar}{{\scriptscriptstyle{(1)**}}}

\newcommand{\re}{\mathrm{re}}

\DeclareMathOperator{\mult}{mult}

\DeclareMathOperator{\Eff}{Eff}

\DeclareMathOperator{\Nef}{Nef}
\DeclareMathOperator{\Amp}{Amp}
\DeclareMathOperator{\NefE}{Nef^e}
\DeclareMathOperator{\Comp}{Comp}
\DeclareMathOperator{\Cr}{Cr}
\DeclareMathOperator{\GL}{GL}
\DeclareMathOperator{\Aut}{Aut}
\DeclareMathOperator{\Int}{Int}
\DeclareMathOperator{\Supp}{Supp}

\DeclareMathOperator{\nod}{nod}
\DeclareMathOperator{\gen}{gen}
\DeclareMathOperator{\irr}{irr}

\DeclareMathOperator{\Ker}{Ker}

\begin{document}
	\title[Generalised Cone Conjecture, III: Non-generic Sakai surfaces]{Generalised Cone Conjecture, III:\\ Non-generic Sakai surfaces}
 
	\author[V.\ Lazi\'c]{Vladimir Lazi\'c}
	\address{Fachrichtung Mathematik, Campus, Geb\"aude E2.4, Universit\"at des Saarlandes, 66123 Saarbr\"ucken, Germany}
	\email{lazic@math.uni-sb.de}

        \author[I.\ Stenger]{Isabel Stenger}
        \address{Institute of Algebraic Geometry, Leibniz University Hannover, Welfengarten 1, 30167 Hannover, Germany}
        \email{stenger@math.uni-hannover.de}
		
	\author[Z.\ Xie]{Zhixin Xie}
	\address{Institut \'Elie Cartan de Lorraine, Universit\'e de Lorraine, 54506 Nancy, France}
	\email{zhixin.xie@univ-lorraine.fr}
	
	
	\begin{abstract}
    We show that no meaningful version of the Cone Conjecture holds for non-generic Sakai surfaces.
		
	\end{abstract}

	\maketitle
	
	\begingroup
		\hypersetup{linkcolor=black}
		\setcounter{tocdepth}{1}
		\tableofcontents
	\endgroup

\section{Introduction}
In \cite{LSX26a,LSX26b} we proved the Generalised Cone Conjecture in dimension two. One of the main steps of the proof is the following result \cite[Theorem C]{LSX26b}, which shows that generic Sakai surfaces satisfy this conjecture.

\setcounter{thmA}{2}
\begin{thmA}\label{thm:main_thm_generic}
  Let $X$ be a  generic Sakai surface, and let $D$ be the unique effective divisor in $|{-}K_X|$. Then there exists a rational polyhedral fundamental domain for the action of the group $\Cr(X,D)$ on $\Nef^e(X)$.
\end{thmA}

Recall that a Sakai surface is a particular blowup of $\pr^2$ at nine (possibly infinitely near) points, and it has a unique divisor $D \in |{-}K_X|$, see  Definition \ref{dfn:goodHalphen}. Generic Sakai surfaces are those Sakai surfaces which do not have any $({-}2)$-curves disjoint from the support of the divisor $D$; we call these curves \emph{internal $(-2)$-curves}. The group $\Cr(X,D)$ is a subgroup of the group of Cremona isometries of $X$ preserving the classes of the components of $D$, see Definition \ref{dfn:cremona}.

In \cite[Proposition D]{LSX26b} we constructed an example of non-generic Sakai surface with exactly one internal $(-2)$-curve, such that there exists no rational polyhedral fundamental domain for the action of $\Cr(X,D)$ on $\NefE(X)$. In this paper, we show that this example is not a sporadic one. In fact, for {\bf every} non-generic Sakai surface, no version of the Cone Conjecture with respect to a geometric meaningful group holds. This is the main result of the present paper.

\begin{thmA}\label{thm:main_thm}
    Let $X$ be a non-generic Sakai surface. Let $G\subseteq\GL\big(N^1(X)_\R\big)$ be any group which preserves the intersection form on $N^1(X)_\R$ and the cone $\Eff(X)$. Then there does not exist a rational polyhedral fundamental domain for the action of $G$ on $\NefE(X)$.
\end{thmA}

A main ingredient in the proof of Theorem \ref{thm:main_thm} is the fact, proved in Theorem \ref{thm:no_fund_dom_non_generic} below, that the action of the group $\Cr(X,D)$ on the set of classes of $({-}1)$-curves on $X$ has infinitely many orbits. This behaviour is completely opposite to that in the generic case \cite[Theorem 6.1]{LSX26b}.   

Although no variant of the Cone Conjecture holds for non-generic Sakai surfaces,  we are still able to compare the nef cone of a non-generic Sakai surface $X$ with that of a generic Sakai surface $X_{\gen}$ of the same type. More precisely, using an isomorphism between the N\'eron--Severi spaces of $X$ and $X_{\gen}$, we show in Theorem \ref{thm:nef_fund_domain_for_nefgen} that $\Nef(X)$ is a fundamental domain for the action of the finite Weyl group $W^{\nod}$ on $\Nef(X_{\gen})$, where $W^{\nod}$ is generated by the reflections associated to the internal $(-2)$-curve classes. 

\medskip

\noindent{\sc Comparison with the literature}

\medskip

In the following, we explain why the simultaneous presence of Theorems \ref{thm:main_thm_generic} and \ref{thm:main_thm} is surprising in comparison with other related results. 

In light of the proof of the movable Cone Conjecture for hyperkähler manifolds in \cite[Theorem 6.25]{Mar11}, one might expect that the situation of Sakai surfaces is similar to the case of K3 surfaces, and that Theorem \ref{thm:nef_fund_domain_for_nefgen}
together with \cite[Theorem C]{LSX26b} implies the existence of a rational polyhedral fundamental domain of $\NefE(X)$ under the action of certain group of geometric origin. However, this is not the case: the main reason is that the group $W^{\nod}$ mentioned above is never a normal subgroup of the group of Cremona isometries of $X_{\gen}$, which we show in Lemma \ref{lem:not_normal}. This is contrary to the case of K3 surfaces, see \cite[Corollary 6.9 and Example 6.7]{Mar11}.

Moreover, Theorems \ref{thm:main_thm_generic} and \ref{thm:main_thm} illustrate how the Generalised Cone Conjecture behaves differently under deformation than the classical Morrison--Kawamata Cone Conjecture. 
In higher dimension, \cite[Theorem 1.1]{Lut24} showed that the Cone Conjecture is invariant under deformation for smooth Calabi--Yau threefolds. However, for klt Calabi--Yau generalised pairs, already in dimension two it is necessary to work only with \emph{very general} objects. Indeed, by \cite[Section 5 and Appendix B]{Sak01} and \cite[Proposition B.5]{LSX26a}, a generic Sakai surface is a very general fibre in an analytic family of Sakai surfaces, parametrised by an analytic variety. By Theorem \ref{thm:main_thm_generic}, the Generalised Cone Conjecture holds for such a fibre. Now Theorem \ref{thm:main_thm} clarifies that no meaningful version of the conjecture holds on other fibres of the family. 

\subsection*{Declaration}

All results in this paper were established by the three authors.

\subsection*{Acknowledgements} 

We are deeply indebted to Auguste H\'ebert for continuous help on the theory of root systems and Weyl groups, and for guiding us through the vast literature. We would like to thank C\'ecile Gachet and Wendelin Lutz for many useful discussions and comments over several years on the content of the present paper. We are grateful to Serge Cantat, Pierre-Emmanuel Chaput and Julia Schneider for many valuable discussions. This project started in 2023 when all three authors were based at the Universit\"at des Saarlandes.

Lazi\'c gratefully acknowledges support by the Deutsche Forschungsgemeinschaft (DFG, German Research Foundation) – Project-ID 286237555 – TRR 195 and Project-ID 530132094. The research of Stenger was partly conducted within the framework of the DFG-funded research training group RTG 2965: From Geometry to Numbers, Project number 512730679.
Xie was partly supported by PEPS JCJC de l'Insmi 2025.

\section{Preliminaries}\label{sec:prelim}

Throughout the paper, we work with varieties defined over $\C$. Most of the notation we need in this paper was introduced in \cite{LSX26a,LSX26b}; here we recall the most important properties that we need in this work.

If $X$ is a normal projective variety, we denote by $N^1(X)_\R$, $\Nef(X)$, $\Amp(X)$, $\Eff(X)$ and $\overline{\Eff}(X)$ the N\'eron-Severi space and the nef, ample, effective and pseudoeffective cone of $X$, respectively. If $D$ is an $\R$-Cartier $\R$-divisor on $X$, then we denote its numerical class in $N^1(X)_\R$ by $[D]$. We call the cone $\Nef^e(X)\coloneqq\Nef(X)\cap\Eff(X)$ the \emph{effective nef cone}. The \emph{rational nef cone} $\Nef^+(X)$ is the cone in $N^1(X)_\R$ spanned by the classes of nef line bundles. 

We fix the following notation throughout the paper. 

\begin{notation}\label{not:MXNX}
	Let $X$ be a smooth projective surface. Assume that there exists a divisor $D\in|{-}K_X|$. Let $H$ be an ample $\Q$-divisor on $X$. We define
	\begin{align*}
       \mathcal C_X^+ & \coloneqq  \{ x\in N^1(X)_\R\mid x^2 >0  \text{ and } x\cdot H >0 \}, \\
		\Comp(D) &\coloneqq \big\{ \cal{D}_i\in N^1(X) \mid \cal{D}_i \text{ is the class of a component of } D\big\},\\
		\cal{M}_X &\coloneqq  \big\{ [E]\in N^1(X) \mid E^2 = K_X\cdot E = -1 \text{ and }h^0(X,E)> 0 \big\},\\
		\cal{M}_X^{\irr} &\coloneqq \big\{ [E]\in N^1(X) \mid E \text{ is a }(-1)\text{-curve} \big\}, \\
		\mathcal{N}_X &\coloneqq \big\{ E\in \mathcal M_X \mid E\cdot C \geq 0 \text{ for all } C \in \Comp(D) \big\},\\
		\Delta_{X,D}^{\nod}&\coloneqq\big\{[C] \in N^1(X)\mid C\text{ is a $(-2)$-curve with}\ C\cap\Supp D=\emptyset\big\}.
	\end{align*}
The set $\Delta_{X,D}^{\nod}$ is well defined by \cite[Lemma 2.4]{LSX26a}, and a curve whose class is in $\Delta_{X,D}^{\nod}$ is called an \emph{internal $(-2)$-curve}.
\end{notation}

\subsection{Cremona isometries}
In this subsection, we recall the definition of Cremona isometries from \cite[Definition 2.16]{LSX26a}. We restrict to the case of smooth surfaces, see also \cite{Loo81,Dol83}.

\begin{dfn}\label{dfn:cremona}
	Let $X$ be a smooth projective surface. An automorphism $\sigma$ of the vector space $N^1(X)_\R$ is a \emph{Cremona isometry} of $X$ if it satisfies the following properties:
	\begin{enumerate}[\normalfont (i)]
		\item $\sigma$ maps the lattice $N^1(X)$ onto itself,
		\item $\sigma$ preserves the intersection form on $N^1(X)_\R$,
		\item $\sigma$ fixes the class $[K_X]$,
		\item $\sigma$ preserves the effective cone $\Eff(X)$.
	\end{enumerate}
	We denote the group of Cremona isometries of $X$ by $\Cr(X)$. Additionally, if $D$ is an effective divisor on $X$ with irreducible components $D_1,\dots,D_m$, 
    we define $\Cr(X,D)$ as
	\[\Cr(X,D)\coloneqq\big\{\sigma\in\Cr(X)\mid\sigma \big([D_i]\big)=[D_i]\text{ for } i=1,\ldots,m\big\}.\]
\end{dfn}

\subsection{Sakai surfaces}
In this paper, we are mainly interested \emph{Sakai surfaces}. We recall their definition in this subsection.

Let $D = \sum n_iD_i$ be an effective integral divisor on a smooth projective surface $X$, where $D_i$ are prime components. Then $D$ is \emph{of canonical type} if $K_X\cdot D_i = D\cdot D_i = 0$ for all $i$. 

\begin{dfn}\label{dfn:goodHalphen}
		A smooth projective rational surface $X$ is called  a \emph{Sakai surface} if the linear system $|{-}K_X|$ contains a divisor $D$ of canonical type and $\kappa(X,-K_X) = 0$. We call $X$ is \emph{generic} if $\Delta^{\nod}_{X,D} = \emptyset$, and \emph{non-generic} otherwise. 
\end{dfn} 

\begin{lem}\label{lem:Cr(X,D)_preserves_Mirr_internal}
    Let $X$ be a Sakai surface and let $D \in |{-}K_X|$. Then $\Cr(X,D)$ preserves the sets $\Delta_{X,D}^{\nod}$ and  $\mathcal{M}^{\irr}_X$.
\end{lem}

\begin{proof}
Fix $w \in \Cr(X,D)$. We first show that $w$ preserves $\Delta^{\nod}_{X,D}$. Let $\gamma = [C] \in \Delta_{X,D}^{\nod}$, where $C$ is an internal $({-}2)$-curve. Recall that $\Delta_{X,D}^{\nod}$ is a finite set by \cite[Lemma 2.7]{LSX26b}.  Then $w$ permutes the $({-}2)$-curves on $X$ (see, e.g.\ \cite[Proof of Proposition D, Step 4]{LSX26b}). Moreover, for each $\cal D \in \Comp(D)$, we have 
    $$
    w(\gamma) \cdot \cal D  =  w(\gamma) \cdot w(\cal D ) = \gamma \cdot \cal D =0.
    $$
    Thus, $w(\gamma)$ is the class of a $({-}2)$-curve which is disjoint from $\Supp(D)$, which shows $w(\gamma) \in \Delta_{X,D}^{\nod}$ as desired. 
    In particular, there exists an element $s \in S_{|\Delta_{X,D}^{\nod}|}$ such that 
    \begin{equation}\label{eq:permutation_internal_curves}
         w(\gamma_i) = \gamma_{s(i)} \text{ for every } \gamma_i \in\Delta_{X,D}^{\nod}.
    \end{equation}

Now we show that $w$ preserves $\mathcal{M}^{\irr}_X$.
    Let $E \in  \mathcal{M}^{\irr}_X $. Since $w \in \Cr(X,D)$ preserves the intersection pairing and $[-K_X]$, we have $w(E) \in \mathcal{M}_X$.  Moreover,  for each $\cal D \in \Comp(D)$, since $w$ fixes $\cal{D}$, we have
    $$
    w(E) \cdot \cal D = w(E) \cdot w(\cal D ) = E \cdot \cal D \geq 0, 
    $$
 and for each $\gamma_i \in \Delta_{X,D}^{\nod}$, it follows from \eqref{eq:permutation_internal_curves} that there exists $s' \in S_{|\Delta_{X,D}^{\nod}|}$ such that
  $$
    w(E) \cdot \cal \gamma_i = E \cdot w^{-1}(\gamma_i) = E \cdot \gamma_{s'(i)} \geq 0.
    $$
  Therefore, using \cite[Lemma 2.6]{LSX26b} we conclude that $w(E) \in \mathcal{M}^{\irr}_X$, as desired. 
\end{proof}

\subsection{Root system on Sakai surfaces}\label{subsec:tauto_basis} 
In \cite[Section 3]{LSX26b} we discussed in detail the theory of root systems, following \cite{Kac90, Bou08}, and the notion of root data introduced in \cite{MP89}. In this subsection we briefly recall the root systems associated with a Sakai surface and we refer to \cite[Section 4]{LSX26b} for more details. 

Throughout the section, $X$ is a Sakai surface with the unique divisor $D\in|{-}K_X|$. We denote by $\mathcal{D}_i$ the classes of the components of $D$.

Recall from \cite[Lemma 2.9(b)]{LSX26a} that $X$ is the blowup of $\pr^2$ at $9$ (possibly infinitely near) points. Then there exist birational morphisms $\pi_i\colon X_i\to X_{i-1}$ for $i=1,\dots,9$, such that $X_9:=X$ and $X_0:=\pr^2$, and each $\pi_i$ is a contraction of a $(-1)$-curve $E_i\subseteq X_i$. We call the induced morphism $\pi_1\circ\ldots\circ\pi_9\colon X\to\mathbb{P}^2$ a \emph{blowdown structure} on $X$.

Let $\mathcal E_0\in N^1(X)_\R$ be the class of the pullback of a line on $\pr^2$, and let $\mathcal E_i\in N^1(X)_\R$ be the class of the pullback of $E_i$ for any $i\in\{1,\dots,9\}$. We call $\{\mathcal{E}_0,\ldots,\mathcal{E}_9\}$ the  \textit{tautological basis} of $N_1(X)_\R$.
It is easy to check that
\begin{equation}\label{eq:intersectnumbers}
    \mathcal E_0^2=1, \quad \mathcal E_i^2={-}1 \text{ for }i>0,\quad \mathcal E_i\cdot\mathcal E_j= 0 \text{ for }i\neq j,
\end{equation}
and
\begin{equation}\label{eq:anticandiv}
    -K_X \sim 3\mathcal E_0 - \mathcal E_1-\dots-\mathcal E_9.
\end{equation}
Moreover,
\begin{equation}\label{eq:picgroup}
N^1(X) = \Z \mathcal E_0\oplus\dots\oplus \Z \mathcal E_9.
\end{equation}

We consider the following classes in $N^1(X)_\R$:
$$ \rho_8 \coloneqq \mathcal E_0 - \mathcal E_1 - \mathcal E_2 - \mathcal E_3, \quad \rho_i \coloneqq \mathcal E_i-\mathcal E_{i+1}\text{ for }i=1,\dots, 7,\quad \rho_0 \coloneqq \mathcal E_8-\mathcal E_9;$$
Then the set
\begin{equation}\label{eq:Pi_E_8}
    \Pi(E_8^\smone)\coloneqq \{\rho_0,\rho_1,\dots,\rho_8\}\subseteq N^1(X)_\R 
\end{equation} 
is linearly independent in $N^1(X)_\R$. We consider the pairing
\begin{equation}\label{eq:pairing}
\langle\cdot\, ,\cdot\rangle\colon N^1(X)_\R\times N_1(X)_\R\to\R,\quad \langle\alpha,\beta\rangle\mapsto{-}\alpha\cdot \beta,
\end{equation}
and we set 
$$\rho_i^\smvee\coloneqq \rho_i\in N_1(X)_\R\quad\text{for all }i.$$
Set
$$A\coloneqq \big(\langle \rho_i,\rho_j^\smvee\rangle\big)_{0\leq i,j\leq8}.$$
Then $A$ is a generalised Cartan matrix and  $\big(N^1(X)_\R,\Pi(E_8^\smone),\Pi(E_8^\smone)^\smvee\big)$ is its realisation. The Dynkin diagram associated to $A$ is of type $E_8^\smone$, see \cite[p.\ 44]{Kac90}. In particular, $A$ is indecomposable and of affine type. Additionally, if we denote
$$\delta\coloneqq [{-}K_X],$$
then using \eqref{eq:anticandiv}, one can check that $\delta$ generates the set of imaginary roots of $\big(N^1(X)_\R,\Pi(E_8^\smone),\Pi(E_8^\smone)^\smvee\big)$, see \cite[Theorem 3.2(b)]{LSX26b}. Note that
$$ \delta^2=0\quad\text{and}\quad\delta\cdot\rho_i=0\text{ for all }i. $$
We denote by $Q(E_8^\smone)$ the lattice spanned by $\Pi(E_8^\smone)$. By \cite[Propositions 7 and 8]{Sak01} we have $Q(E_8^\smone)=\{x\in N^1(X)\mid x\cdot\delta=0\}$. We denote by $W(E_8^\smone)$ the associated Weyl group and by $\Delta(E_8^\smone)^\re$ the associated set of real roots. Since every element of $Q(E_8^\smone)$ is orthogonal to $\delta$, we have
$$ w(\delta)=\delta\quad\text{for any }w\in W(E_8^\smone). $$
Moreover, for every $w \in W(E_8^\smone)$ and $x,y \in N^1(X)_{\R}$, we have
\begin{equation}\label{eq:preserve_pairing}
w(x) \cdot y = x \cdot w^{-1}(y).
\end{equation}
	  	
We recall from \cite[\S 4.4]{LSX26b} the definition of the \emph{negative fundamental chamber}
        $$ \mathcal F(E_8^\smone) = \big\{x\in N^1(X)_\R \mid x\cdot\beta > 0 \text{ for all }\beta\in \Pi(E_8^\smone)^\smvee\big\} \subseteq N^1(X)_\R, $$
and the \emph{negative Tits cone}
        $$ \mathcal T(E_8^\smone) = \bigcup_{w\in W(E_8^\smone)} w\big(\overline{\mathcal{F}(E_8^\smone)}\big)\subseteq N^1(X)_\R. $$
Moreover,  we have 
\begin{equation}\label{eq:tautological_tits_cone}
\mathcal T(E_8^\smone) = \{x\in N^1(X)_{\R}\mid \delta\cdot x>0 \}\cup\R[\delta].
\end{equation}

We consider special subroot systems of the root system $E_8^\smone$ above, following \cite{Sak01}. The \emph{root system $R$} is a root system with the basis
$$ \Pi(R)\coloneqq \{\mathcal{D}_i\mid \mathcal{D}_i \text{ is the class of a component of } D\}. $$
This is indeed a root system by \cite[p.\ 181]{Sak01}, and it depends on the blowdown structure of $X$. 

The surfaces $X$ are classified in according to the \emph{type $R$} of $D$. All possible subroot systems $R$ are given in \cite[pp.\ 181--184]{Sak01}:
\begin{align*}
&A_0^\smone,\ A_1^\smone,\ A_2^\smone,\ A_3^\smone,\ A_4^\smone,\ A_5^\smone,\ A_6^\smone,\ A_7^\smone,\ {A_7^\smone}',\\
&A_8^\smone,\ D_4^\smone,\ D_5^\smone,\ D_6^\smone,\ D_7^\smone,\ D_8^\smone,\ E_6^\smone,\ E_7^\smone,\ E_8^\smone.
\end{align*}
Note that the four additional items in the Sakai's list, $A_0^\smonestar$, $A_0^\smonestarstar$, $A_1^\smonestar$ and $A_2^\smonestar$, give the same root system $R$ as of type $A_0^\smone$, $A_1^\smone$ and $A_2^\smone$, respectively. Hence they carry no new numerical information; the $*$ refers to different geometric configurations of the components of $D$.

We set $\mathcal{D}_i^\smvee \coloneqq \mathcal{D}_i\in N_1(X)_{\R}$ for each class of a component of $D$, and set
$$ A\coloneqq \big( \langle \mathcal{D}_i,\mathcal{D}_j^\smvee \rangle \big)_{i,j}. $$
By \cite[Lemma 3.11]{LSX26b}, if $R\neq A_0^\smone$, then
$$
\big(A,\Pi(R),\Pi(R)^\smvee,N^1(X)_\R,N_1(X)_\R,\langle\cdot\, ,\cdot\rangle\big)
$$
is a root datum of affine type.

In the remainder of this subsection, we assume that  $ R\notin\{A_8^\smone, D_8^\smone, E_8^\smone \}$.
Consider the lattice
$$Q(R)^{\smperp} \coloneqq \{v \in N^1(X) \mid v \cdot  \cal{D}_i =0 \text{ for all }i \} $$
and the set $\Pi(R^\smperp)$ from \cite[Appendix A]{Sak01}. We denote by $Q(R^{\smperp})$ the lattice spanned by $\Pi(R^\smperp)$.
Observe that $Q(R)\cap Q(R^{\smperp}) = \Z \delta$. Moreover, we have $Q(R^{\smperp}) = Q(R)^{\smperp}$ when $R$ is different from $A_6^{\smone}$.
We consider the pairing $\langle\cdot\, ,\cdot\rangle$ as in \eqref{eq:pairing}.
Then by  \cite[Proposition 4.1(a)]{LSX26b}, there exists a generalised Cartan matrix $A$ such that
\begin{equation}\label{eq:set_root_data_Rperp}
    \big(A,\Pi(R^\smperp),\Pi(R^\smperp)^\smvee,N^1(X)_\R,N_1(X)_\R,\langle\cdot\, ,\cdot\rangle\big)
\end{equation}
is a root datum with associated 
associated Weyl group $W(R^\smperp)$ and real roots
$\Delta(R^\smperp)^{\re}$.
We denote by $\mathcal F(R^\smperp)$and $\mathcal T(R^\smperp)$
its negative fundamental chamber and its negative Tits cone. 

Assume now that
 $R\notin\{A_5^\smone,A_6^\smone, D_6^\smone \}$.
Then $A$
is an indecomposable symmetric generalised Cartan matrix of affine type, and thus \eqref{eq:set_root_data_Rperp} is a root datum of affine type, see \cite[\S 4.7]{LSX26b}.

Finally, assume that 
$ R\in\{A_5^\smone,A_6^\smone, D_6^\smone \}$. Then, by \cite[\S 4.8]{LSX26b}, the matrix $A$ is decomposable and the root datum \eqref{eq:set_root_data_Rperp} decomposes into root data $\mathscr{D}$ and $\mathscr{D}'$ of affine type. In particular, we have 
\begin{equation}\label{eq:decompose_W}
    W(R^\smperp) = W_{\mathscr{D}}W_{\mathscr{D}'}.
\end{equation}

\subsection{Group actions and fundamental domains}

Let $\Lambda$ be a free $\Z$-module of finite rank and let $V\coloneqq \Lambda\otimes_\Z \R$. A \emph{convex cone} in $V$ is a convex subset $\mathcal C\subseteq V$ such that $\lambda v\in\mathcal C$ for every $\lambda\geq0$ and $v\in\mathcal C$. 
It is \emph{strictly convex} if its closure $\overline{\mathcal C}$ contains no lines. 
It is \emph{polyhedral}, respectively \emph{rational polyhedral}, if it is spanned by finitely many elements of $V$, respectively of $\Lambda$.
For a convex cone $\mathcal C\subseteq V$ we denote by $\mathcal \Int(C)$ the interior of $\mathcal C$, and by $\mathcal C^+$ the convex hull of $\overline{\mathcal C}\cap \Lambda$. We say that a subgroup of $G\subseteq\GL(V)$ \emph{preserves $\mathcal C$} if $g(\mathcal C)=\mathcal C$ for each $g\in G$; then we also say that \emph{$G$ acts on $\mathcal C$}. Moreover, if $\Pi\subseteq\mathcal C$, then we denote $G\cdot\Pi\coloneqq\bigcup_{g\in G}g(\Pi)$.

\begin{dfn}\label{dfn:polyhedral_type}
	Let $\mathcal C\subseteq V$ be a strictly convex cone such that $\Int(\mathcal C)\neq\emptyset$, and let $G$ be a subgroup of $\GL(V)$ which preserves $\Lambda$ and $\mathcal C$. 
	We say that the action of $G$ on $\mathcal C$ 
    \begin{enumerate}[\normalfont (a)]
\item  is \emph{of polyhedral type} if there exists a polyhedral cone $\Pi \subseteq\mathcal C^+$ such that $\operatorname{Int}(\mathcal C)\subseteq G \cdot \Pi$;
    \item  admits a \emph{rational polyhedral fundamental domain} if there exists a rational polyhedral cone $\Pi \subseteq \mathcal C$ such that $G \cdot \Pi = \mathcal C$, and if $g\big( \Int(\Pi) \big) \cap \Int(\Pi) \neq\emptyset$ for some $g \in G$, then $g$ is the identity in $G$.
    \end{enumerate}
\end{dfn}

\subsection{An auxiliary lemma} We need the following lemma in the proof of Theorem \ref{thm:no_fund_dom_non_generic}.

\begin{lem}\label{lem:finite_orbits}
    Let $(G,\cdot)$ be a group acting on a set $S$, and let $H\subseteq G$ be a subgroup of finite index. Assume that $G$ acts on $S$ with finitely many orbits. Then the action of $H$ on $S$ has finitely many orbits as well.
\end{lem}

\begin{proof}
    By assumption, there exists a finite index set $I$ and elements $g_i\in G$ such that $\{ Hg_i \}_{i\in I}$ are all the right cosets of $H$ in $G$. By assumption, there exists a finite index set $J$ such that there are precisely $|J|$ orbits of $G$ on $S$, and let $\{ s_j \}_{j\in J}$ be elements representing distinct orbits.
    
    Let $s\in S$. Then there exist $j\in J$ and $g\in G$ such that $s= g(s_j)$, and there exist $i\in I$ and $h\in H$ such that $g= hg_i$. Then $ s= (hg_i) s_j = h \big( g_i (s_j)\big )$, thus $s$ belongs to the $H$-orbit of $g_i (s_j)$. Therefore, $H$ has at most $|I|\cdot |J|$ orbits, as desired.
\end{proof}

\section{Nodal Weyl group}\label{sec:nodal_weyl}

Let $X$ be a Sakai surface and let $D$ be the unique element in $|{-}K_X|$. Recall from \cite[Lemma 2.7]{LSX26b} that the classes of $({-}2)$-curves in $N^1(X)_{\R}$ are linearly independent. In particular, $X$ has at most nine $({-}2)$-curves. Therefore, we we either have $ \Delta_{X,D}^{\nod} = \emptyset$, i.e.\ $X$ is generic, or there exist a positive integer $m \leq 9$ and classes $\gamma_1,\dots,\gamma_m\in N^1(X)$ of $(-2)$-curves on $X$ such that 
$$ \Delta_{X,D}^{\nod}= \{ \gamma_1,\dots,\gamma_m \}. $$
We associate a root datum to $\Delta_{X,D}^{\nod}$: for each $i\in \{ 1,\dots,m\}$, we define $ \gamma_i^\smvee \coloneqq \gamma_i \in N_1(X) $ and we set
$$ \Delta_{X,D}^{\nod,\smvee} = \{ \gamma_1^\smvee,\dots,\gamma_m^\smvee \}. $$
Then the associated matrix 
$$ A^{\nod} \coloneqq \big( \langle \gamma_i, \gamma_j^\smvee \rangle \big)_{1\leq i,j\leq m} $$
is clearly a symmetric generalised Cartan matrix.

\begin{lem}\label{lem:nodal_root_data}
    Let $X$ be a non-generic Sakai surface and let $D\in |{-}K_X|$. Then
    $$\big( A^{\nod},\Delta_{X,D}^{\nod},\Delta_{X,D}^{\nod,\smvee}, N^1(X)_\R, N_1(X)_\R, \langle \cdot , \cdot \rangle \big) $$
    is a root datum.
\end{lem}

\begin{proof}
The condition (i) in \cite[Definition 3.6]{LSX26b} follows from the construction. Since the elements in $\Delta_{X,D}^{\nod}$ are linearly independent, the conditions (ii) and (iii) in \cite[Definition 3.6]{LSX26b} are verified.
\end{proof}

In view of Lemma \ref{lem:nodal_root_data}, we can define the Weyl group
$$ W^{\nod}\coloneqq \langle r_{\gamma_i} \mid \gamma_i \in \Delta_{X,D}^{\nod} \rangle $$
and the set of real roots
$$ R^{\nod} \coloneqq W^{\nod} \cdot \Delta_{X,D}^{\nod}. $$
Note that for any $x\in N^1(X)$ and for any $i\in \{1,\dots,m\}$ we have
$$r_{\gamma_i}(x)= x + (\gamma_i\cdot x) \gamma_i \in N^1(X).$$
Thus, $W^{\nod} \subseteq \Aut \big(N^1(X)\big)$. We define the \emph{negative fundamental chamber of }$\Delta^{\nod}_{X,D}$ as
\begin{align*}
\cal{F}^{\nod} &\coloneqq \{ x\in N^1(X)_{\R} \mid \langle x, \gamma^{\smvee}_i \rangle < 0 \text{ for any }\gamma_i\in \Delta^{\nod}_{X,D}\} \\
&= \{ x\in N^1(X)_{\R} \mid x\cdot \gamma_i  > 0 \text{ for any }\gamma_i\in \Delta^{\nod}_{X,D}\},
\end{align*}
and the associated \emph{Tits cone}
$$\cal{T}^{\nod} \coloneqq W^{\nod} \cdot \overline{\cal F^{\nod}}.$$

Next we show that $R^{\nod}$ is a subroot system of $E_8^\smone$ of finite type.

\begin{prop}\label{prop:nodel_root_system}
    Let $X$ be a Sakai surface. Then the following holds:
    \begin{enumerate}[\normalfont (a)]
        \item $R^{\nod}$ is a subroot system of $\Delta(E_8^{\smone})^{\re}$,
       \item $A^{\nod}$ is positive definite,
       \item the sets $R^{\nod}$ and $W^{\nod}$ are finite.
    \end{enumerate}
\end{prop}

\begin{proof}
We will use throughout the proof that the numerical and linear equivalence coincide on $X$, since $X$ is a rational surface.

    Part (a) follows immediately from \cite[Definition 3.14]{LSX26b} and Lemma \ref{lem:nodal_root_data} and by \cite[Proposition 18(a)]{Sak01}.

  For (b), let $D\in|{-}K_X|$. After reordering the indices, we may decompose the generalised Cartan matrix $A^{\nod}$ into a direct sum of indecomposable matrices 
  $$A^{\nod} = A_1 \oplus \cdots \oplus A_k.$$
  Fix $j \in \{1,\ldots,k\}$, and let $n_j$ denote the size of $A_j$. Without loss of generality, we may assume that
  $$\Pi_j^{\nod} =  \{\gamma_1,\ldots,\gamma_{n_j} \}\subseteq \Delta_{X,D}^{\nod}$$
  is the set of curve classes defining the matrix $A_j$. It suffices to show that $A_j$ is positive definite.
  
  For each $i\in\{1,\dots,n_j\}$, let $C_i$ be the internal $({-}2)$-curve such that $\gamma_i=[C_i]$. For each $ \mathbf{x} \coloneqq  (x_1,\ldots, x_{n_j})\neq 0 $ set 
  $$E_\mathbf{x} \coloneqq \sum_{i=1}^{n_j} x_iC_i.$$ 
  Since the classes in $\Delta_{X,D}^{\nod}$ are orthogonal to $[K_X]$, we have $[E_\mathbf{x}]\cdot[K_X]=0$. Since $K_X^2=0$ and $[K_X]\neq0$, by the Hodge index theorem we infer that
    \begin{equation}\label{eq:semidef}
  	E_\mathbf{x}^2\leq 0. 
    \end{equation}
	Therefore, $ \mathbf{x} A_j \mathbf{x}^T = - E_\mathbf{x}^2 \geq 0 $, thus the matrix $A_j$ is positive semidefinite.

To prove that $A_j$ is positive definite, it is enough to show that $A_j$ has full rank. Assume otherwise. Then there exists a vector $\mathbf{y} \neq 0$ such that  $A_j \mathbf{y}^T = 0$. Since $A_j$ has integral coefficients, we can choose $\mathbf{y}$ to be a vector with integral entries. Then $E_\mathbf{y}$ is an integral divisor, and $E_\mathbf{y}^2 = {-}\mathbf{y} A_j \mathbf{y}^{\textrm{T}} = 0$. Therefore, by \cite[Lemma 2.2]{LH05} we have
$$[E_\mathbf{y}] = \big[r(-K_X)\big]=[rD]\quad \text{for some integer }r.$$
By replacing $\mathbf{y}$ with ${-}\mathbf{y}$ if necessary, we may assume that $r\geq0$.

Write $E_\mathbf{y}=E_1-E_2$, where $E_1$ and $E_2$ are effective divisors without common components. Then
\begin{equation}\label{eq:E1E2}
E_1\sim rD+E_2.
\end{equation}
Since $E_1$ is supported on the curves $C_1,\ldots, C_{n_j}$, which are disjoint from the components of $D$, the divisor $E_1$ has no common components with $rD+E_2$, hence by \eqref{eq:E1E2} it intersects every curve on $X$ non-negatively. Thus, $E_1$ is nef.
Moreover, it follows from \eqref{eq:semidef} that
$E_1^2\leq0$. Therefore, $E_1^2=0$, hence by \cite[Lemma 2.2]{LH05} we have
$$[E_1] = \big[s(-K_X)\big]=[sD]\quad \text{for some non-negative integer }s.$$
As $\kappa(X,D)=0$, this implies that $E_1=sD$, hence $E_1=0$ since $E_1$ and $D$ have no common components by above. Then $rD+E_2\sim0$ by \eqref{eq:E1E2}, thus $E_2=0$. Therefore, $E_\mathbf{y}=0$, which is a contradiction that shows (b).

Finally, to show (c), let $V_j\subseteq N^1(X)_{\R}$ be the vector subspace spanned by $\gamma_1,\ldots,\gamma_{n_j}$. Then by \cite[Lemma 4.5]{Kac90}, each $(V_j, \Pi_j^{\nod}, \Pi_j^{\nod,\smvee})$ is a realisation of an indecomposable generalised Cartan matrix of finite type. Thus by \cite[Proposition 4.9]{Kac90}, the associated root system $R^{\nod}_j$ and the Weyl group $W^{\nod}_j$ are finite. This, together with \cite[Chapitre VI, \S 1, Définition 1 and Proposition 5]{Bou08}, implies that $R^{\nod}=\bigcup_{j=1}^k R_j^{\nod}$ and $W^{\nod}=\prod_{j=1}^k W_j^{\nod}$ are finite, which proves (c).
\end{proof}

Since $\gamma \cdot \mathcal D =0$ for every $\gamma \in \Delta_{X,D}^{\nod}$ and every $\mathcal{D}\in \Comp (D)$, we have $w(\mathcal{D}) =\mathcal{D}$ for all $w \in W^{\nod}$. Therefore, by Proposition \ref{prop:nodel_root_system}(a) and \eqref{eq:preserve_pairing} we obtain 
$$ w(\gamma)\cdot \mathcal{D} = \gamma \cdot w^{-1}(\mathcal{D}) = \gamma \cdot \mathcal{D} = 0 \quad \text{and}\quad w(\gamma)\cdot w(\gamma) = \gamma^2 = -2. $$
This, together with \cite[Proposition 4.1(d)]{LSX26b}, implies
\begin{equation}\label{eq:R^nod_in_R^perp}
   \Delta^{\nod}_{X,D} \subseteq R^{\nod}\subseteq \Delta(R^\smperp)^{\re}.
\end{equation}
Since $Q(R) \cap Q(R^\smperp) = \Z \delta$ and since $\Z\delta$ is not contained in 
$Q(\Delta^{\nod}_{X,D})$ by the proof of Proposition \ref{prop:nodel_root_system}, it follows from \eqref{eq:R^nod_in_R^perp} that
\begin{equation}\label{eq:direct_sum}
  Q(\Delta^{\nod}_{X,D}) \cap Q(R) = 0. 
\end{equation}

\begin{cor}\label{cor:generic_types}
    Let $X$ be a Sakai surface whose type $R$ is one of $A_7^\smone$, $A_8^\smone$, $D_7^\smone$, $D_8^\smone$ or $E_8^\smone $. 
    Then 
    $X$ is a generic Sakai surface. 
\end{cor}

\begin{proof}
We have a decomposable set of real roots 
$$\Delta(R)^{\re} \cup R^{\nod} \subseteq \Delta(E^{\smone}_8)^{\re}.$$
Since $\Delta(R)^{\re}$ is a subroot system of $\Delta(E^{\smone}_8)^{\re}$, together with Proposition \ref{prop:nodel_root_system}(a) we infer that $\Delta(R)^{\re} \cup R^{\nod}$ is a subroot system of $\Delta(E^{\smone}_8)^{\re}$. In particular, by \eqref{eq:direct_sum}, the lattice $Q(R) \oplus Q(\Delta^{\nod}_{X,D})$  is a sublattice of $Q(E^{\smone}_8)$.

 If $R$ is as in the statement, we infer from \cite[Theorem 3.2]{OS91}
 that there is no root lattice of the form $Q(R \oplus R') = Q(R)
 \oplus Q(R')$ which occurs as sublattice of $Q(E^{\smone}_8)$\footnote{There
   are two possible embeddings of the root lattice $A_7^{\smone}$ in
   $E_8^{\smone}$: the first one corresponds to a primitive embedding,
   while the second one gives an embedding of $A_7^{\smone}\oplus A_1$
   in $E_8^{\smone}$. In the notation from \cite{Sak01} this second
   one is denoted by ${A_7^\smone}'$. Note that in \cite[Theorem
   3.2]{OS91}, each root system in their table is to be interpreted as
   either of finite or of affine type.}. Hence, the previous paragraph implies $\Delta^{\nod}_{X,D} = \emptyset$, which proves the statement. 
\end{proof}

When $X$ is a non-generic Sakai surface of type $R$, we show below that $R^{\nod}$ is actually a subroot system of $\Delta(R^\smperp)^{\re}$.

\begin{lem}\label{lem:nodal_roots} 
   Let $X$ be a non-generic Sakai surface of type $R$ and let $D\in |{-}K_X|$. 
Then:
    \begin{enumerate}[\normalfont (a)]
        \item $\Delta_{X,D}^{\nod}\subseteq \Delta(R^{\smperp})^\re_+$, and in particular, $R^{\nod}$ is a subroot system of $ \Delta(R^\smperp)^{\re}$,
        \item $\overline{\mathcal F(R^\smperp)}\subseteq \overline{\mathcal F^{\nod}}$, 
        \item $\cal{T}(R^{\smperp})\subseteq \cal{T}^{\nod}$.
    \end{enumerate}
\end{lem}

\begin{proof}
Part (b) follows immediately from part (a) and from \cite[Lemma 3.16]{LSX26b}, while part (c) follows from part (a) and \cite[Corollary 3.15]{LSX26b}. We prove part (a) in the following steps.

Throughout the proof we use the following notation: for an element $\gamma\in Q(E_8^\smone)$, if we write $\gamma = \sum_{i=0}^9 a_i\mathcal{E}_i$ with $a_i \in \Z$, then $\mult_{\mathcal E_i}\gamma:=a_i$ for $i=0,\dots,9$. By \cite[Appendix B]{LSX26b} we see that
\begin{equation}\label{eq:Rperp_mult_E0}
\text{for every }\alpha \in\Pi(R^\smperp)\text{ we have }\mult_{\mathcal E_0}\alpha\geq0.
\end{equation}

Fix $\gamma\in \Delta_{X,D}^{\nod}$. Then by \eqref{eq:R^nod_in_R^perp} we have $\gamma \in \Delta(R^{\smperp})^\re$. Since every root is either positive or negative, there exists a subset $\mathcal S\subseteq\Pi(R^\smperp)$ such that
\begin{equation}\label{eq:n_alpha}
    \gamma = \sum_{\alpha\in \mathcal S} n_\alpha \alpha,
\end{equation}
where $n_\alpha\in \Z$ are either all positive or all negative. 
In the remainder of the proof, we will show that there exists an $\alpha\in\mathcal S$ such that $n_\alpha>0$, which will thus prove that $\gamma\in\Delta(R^\smperp)^\re_+$, as desired.

\medskip
\emph{Step 1.} Assume in this step that $\gamma$ is exceptional over $\pr^2$. Since $\cal{E}_0$ is not exceptional over $\pr^2$, by \eqref{eq:Rperp_mult_E0} we have that $\mult_{\cal{E}_0}\alpha = 0$ for every $\alpha\in \cal{S}$. Now by \cite[Appendix A]{Sak01},
 we can rearrange the elements $\alpha \in \cal S $ with $\mult_{\mathcal{E}_0}(\alpha) =0$ so that 
\begin{equation}\label{eq:gamma1}
  \gamma = \sum_
  {k=1}^{\ell} n_k(\mathcal{E}_{r_k} - \mathcal{E}_{s_k}) ,
 \end{equation}
 where $1 \leq r_k < s_k \leq 9$ and 
$ 2 \leq s_1 < \ldots < s_\ell \leq 9$.
In particular, 
\begin{equation}\label{eq:gamma2}
 s_\ell = \max\{ t \mid \mult_{\mathcal{E}_t}(\gamma) \neq 0\}   
\end{equation}
and $\mathcal{E}_{ s_\ell}$ appears exactly once in the sum \eqref{eq:gamma1}.

On the other hand, we can write
$$
\gamma = \sum_{i=1}^9 a_i \cal{E}_i.
$$
   Since $\gamma^2 = -2$, by \eqref{eq:intersectnumbers} we obtain
   $\gamma \in \{ \cal{E}_i + \cal{E}_j, -\cal{E}_i - \cal{E}_j,
   \cal{E}_i - \cal{E}_j, -\cal{E}_i + \cal{E}_j \} $ for some $1\leq
   i<j\leq9 $. We claim that
   \begin{equation}\label{eq:gamma_Ei-Ej}
       \gamma = \cal{E}_i -\cal{E}_j \text{ with }1\leq
   i<j\leq9.
   \end{equation}
   Indeed, $\gamma \neq -\cal{E}_i - \cal{E}_j$ since
   $\gamma$ is effective, and  $\gamma \neq \cal{E}_i + \cal{E}_j$ by \cite[Lemma 2.4]{LSX26a}. Now assume that
   $\gamma = -\cal{E}_i + \cal{E}_j$. Recall that there exist blowups
   $\pi_i\colon X_i\to X_{i-1}$ for $i=1,\dots,9$, such that $X_9:=X$,
   $X_0\coloneqq \pr^2$, and $\pi_1\circ\ldots\circ\pi_9\colon
   X\to\mathbb{P}^2$. Set
   $\varphi\coloneqq \pi_{i+1}\circ\ldots\circ\pi_9$. Then
   $\varphi_*\gamma=-\varphi_*\cal E_i$, a contradiction since both
   $\varphi_*\gamma$ and $\varphi_*\cal E_i$ are effective classes.
   
Therefore, by comparing the coefficients in \eqref{eq:gamma1} with \eqref{eq:gamma_Ei-Ej}, we deduce from \eqref{eq:gamma2} that $j = s_\ell$, and thus 
\[
-n_{\ell} \mathcal{E}_{s_\ell} = - \mathcal{E}_j
\]
 which implies $n_{\ell} = 1$, as desired.

\medskip

\emph{Step 2.}
Assume in this step that $\gamma$ is not exceptional over $\mathbb P^2$. Then, since $\mathcal E_0$ is the pullback of the hyperplane class on $\mathbb P^2$, by the projection formula we obtain that
$$ \mult_{\mathcal E_0}\gamma=\gamma\cdot\mathcal E_0>0. $$
Together with \eqref{eq:Rperp_mult_E0} this shows that there exists $\alpha\in \mathcal S$ such that $n_\alpha >0$, which concludes the proof.
\end{proof}
\begin{rem}\label{rem:the_case_A6}
    If $X$ is a Sakai surface of type $R = A^\smone_6$, then $\Pi(R^\smperp) = \Pi \cup \Pi'$ with associated root data $\mathscr{D}$ and $\mathscr{D'}$ of affine type. 
    By \cite[Proposition 4.1(d)]{LSX26b} we have $\Delta^{\nod}_{X,D} \subseteq \Delta^{\re}_{\mathscr{D}'}$, and thus 
   $W^{\nod} \subseteq W_{\mathscr{D}'}$.
\end{rem}

\section{Infinitely many orbits}

In \cite[Theorem 6.1]{LSX26b} we showed that for a  generic Sakai surface $X$, the action of $\Cr(X,D)$ on the set of classes of $({-}1)$-curves has only finitely many orbits. Surprisingly, the situation is completely reversed in the non-generic case: 

\begin{thm}\label{thm:no_fund_dom_non_generic}
    Let $X$ be a non-generic Sakai surface, and let $D \in |{-}K_X|$. Then:
   \begin{enumerate}[\normalfont (a)]
   \item the action of $\Cr(X,D)$ on $\cal{M}^{\irr}_X$ has infinitely many orbits,
   \item there exists no rational polyhedral fundamental domain for the action of $\Cr(X,D)$ on $\Eff(X)$. 
   \end{enumerate}
\end{thm}

The main part of the proof of the theorem is contained in the following lemma.

\begin{lem}\label{lem_finite_index}
     Let $X$ be a non-generic Sakai surface, and let $D \in |{-}K_X|$. There exists a subgroup $H\subseteq\Cr(X,D)$ of finite index, such that the action of $H$ on $\mathcal{M}^{\irr}_X$ has infinitely many orbits.
\end{lem}

\begin{proof}
We prove this result in several steps. Let $\Delta_{X,D}^{\nod} = \{\gamma_1,\ldots, \gamma_m\}$. 

    \medskip

\emph{Step 1.}
In this step we construct an infinite sequence $\{\mathcal{G}_n\}$ of classes in $\cal{M}_X^{\irr}$. The construction is inspired by \cite[Proposition 3.1]{Lah04}.

By Proposition \ref{prop:nodel_root_system}(b), the intersection matrix of $\{\gamma_1,\ldots,\gamma_m\}$ is negative definite. Thus, it is invertible, and we denote by $d$ its determinant. By Cramer's rule, for each $n \in \Z$ there is a unique set $\{a_{1,n},\ldots,a_{m,n} \}\subseteq \Z^m$ such that 
 $$ \sum_{i=1}^m a_{i,n} \gamma_i\cdot \gamma_j = |d| \cdot (n-\gamma_j \cdot \mathcal{E}_9) \quad \text{for each }j\in\{ 1,\ldots, m\}. $$
Set $\xi_n \coloneqq \sum_{i=1}^m a_{i,n}\gamma_i\in N^1(X)$. Then it is easy to see that $\xi_n^2$ is even, and set
    $$ \mathcal{G}_n \coloneqq \textstyle \mathcal{E}_9 +\xi_n  - (\xi_n \cdot \mathcal{E}_9  + \frac{1}{2}\xi_n^2 )\delta\in N^1(X). $$
    Since $\delta^2 = \xi_n \cdot \delta = 0$ and $\mathcal E_9\cdot \delta=1$, we have
    $$ \mathcal{G}_n^2 =\textstyle -1 + \xi_n^2 + 2 \gamma_n \cdot \mathcal{E}_9 - 2(\xi_n \cdot \mathcal{E}_9  + \frac{1}{2} \xi_n^2 ) = -1 $$
    and 
    $$ \mathcal{G}_n \cdot \delta = \mathcal{E}_9 \cdot \mathcal{G}_n = 1. $$
    Thus, $\mathcal{G}_n \in \cal{M}_X$. 

Set $h \coloneqq \max\{\gamma_j \cdot \mathcal{E}_9\mid 1\leq j\leq m \}$, and fix $n\geq h$. Next we show that for each $n\geq h$ we have $\mathcal G_n\in \cal{M}_X^{\irr}$. By \cite[Lemma 2.6]{LSX26b} it remains to show that $\mathcal G_n\cdot\gamma \geq 0$ for all $\gamma \in \Comp(D) \cup \Delta^{\nod}_{X,D}$. To that end, for each $\gamma\in \Comp (D)$ we have
$$ \mathcal{G}_n  \cdot \gamma = \cal{E}_9 \cdot \gamma \geq 0, $$
where the equality follows since the internal $({-}2)$-curves are disjoint from the components of $D$, and the inequality follows from \cite[Appendix B]{LSX26b}. Further, for $\gamma \in \Delta^{\nod}_{X,D}$ we have
\begin{equation}\label{eq:anequality}
    \mathcal{G}_n  \cdot \gamma = \mathcal{E}_9 \cdot \gamma + \xi_n \cdot \gamma = \mathcal{E}_9 \cdot \gamma + |d|(n -  \mathcal{E}_9\cdot \gamma) \geq 0,
    \end{equation}
where we used that $\gamma\cdot \cal{E}_9\geq 0$ as $\mathcal{E}_9$ is the class of a $({-}1)$-curve and $\gamma$ is the class of a $({-}2)$-curve. 

Finally, for distinct positive integers $n$ and $n'$, we have $\cal G_n \neq \cal G_{n'}$, since $\cal G_n \cdot \gamma \neq \cal G_{n'}\cdot \gamma$ for any $\gamma \in \Delta^{\nod}_{X,D}$ by \eqref{eq:anequality}. Then $\{\mathcal G_n\}_{n\geq h}$ is the desired sequence.

\medskip

\emph{Step 2.}
In this step we finish the proof of the lemma.

Since $\Cr(X,D)$ preserves $ \Delta_{X,D}^{\nod}$ by Lemma \ref{lem:Cr(X,D)_preserves_Mirr_internal}, we obtain a  group homomorphism
    $$ \Psi \colon \Cr(X,D) \to S_{m},\quad w\mapsto s \text{ such that } w(\gamma_i) = \gamma_{s(i)}. $$
    Then $L \coloneqq \Ker \Psi $ is the subgroup of $\Cr(X,D)$ fixing each $\gamma_i$, and $L$ clearly has finite index in $\Cr(X,D)$. 

Assume, for contradiction, that the action of $L$ on $\mathcal{M}^{\irr}_X$ has only finitely many orbits. 
Fix $\gamma\in \Delta^{\nod}_{X,D}$. Then for any $w \in L$ and for any $\beta\in N^1(X)$ we have 
$$ w(\beta) \cdot \gamma = \beta \cdot w^{-1}(\gamma) = \beta \cdot \gamma, $$
where the first equality follows since $L\subseteq \Cr(X,D)$, and the second equality follows from Step 1. Thus, the map $ \mathcal{M}^{\irr}_X  \rightarrow \Z,  \ \beta  \mapsto \beta \cdot \gamma $ is constant on each orbit of $L$, and in particular, since we assume that there are finitely many orbits of $L$ on $\mathcal{M}^{\irr}_X$,
\begin{equation}\label{eq:finite_int_numbers}
    \text{ the set }\{ \beta\cdot \gamma\mid \beta \in  \mathcal{M}^{\irr}_X \} \text{  is finite.}
\end{equation}
On the other hand, by the last paragraph in Step 2, 
the set $\{ \mathcal{G}_n  \cdot \gamma\mid n\geq h \}$ is infinite, which contradicts \eqref{eq:finite_int_numbers} and concludes the proof.
\end{proof}

We can now prove Theorem \ref{thm:no_fund_dom_non_generic}.

\begin{proof}[Proof of Theorem \ref{thm:no_fund_dom_non_generic}]
Part (a) follows from Lemmas \ref{lem:finite_orbits} and \ref{lem_finite_index}. 

For part (b), suppose that there exists a rational polyhedral fundamental domain $\Xi \subseteq \Eff(X)$ for the action of $\Cr(X,D)$ on $\Eff(X)$. We will derive a contradiction in two steps.

\medskip

\emph{Step 1.}
In this step, we show that $\cal{M}^{\irr}_X \cap \Xi$ is finite.

Since $\Eff(X)$ is closed by \cite[Lemma 5.1(a)]{LSX26b}, it follows from \cite[Lemma 2.4]{LSX26a} that any element in $\cal{M}^{\irr}_X$ spans an extremal ray of $\Eff(X)$. Since $\Xi\subseteq \Eff(X)$ is a rational polyhedral cone, we infer that any element in $\cal{M}^{\irr}_X \cap \Xi$ spans an extremal ray of $\Xi$.

Let $C$ and $C'$ be two $(-1)$-curves on $X$ such that $[C'] = \lambda [C]$ for some $\lambda \in \R_{>0}$. Then $ -1 = (C')^2 = \lambda^2 C^2 = - \lambda^2 $, which gives $\lambda =1$ and $[C] = [C']$.
Therefore, there is a bijection between the set $\cal{M}^{\irr}_X$ and the set of extremal rays of $\Eff(X)$ spanned by elements in $\cal{M}^{\irr}_X$. Now since there are finitely many extremal rays of $\Xi$, we conclude that $\cal{M}^{\irr}_X \cap \Xi$ is finite.

\medskip

\emph{Step 2.}
In this step, we show that $\Cr(X,D)$ acts on $\cal{M}^{\irr}_X$ with finitely many orbits, which contradicts part (a) and concludes the proof of part (b).

    Let $\gamma \in \cal{M}^{\irr}_X \subseteq \Eff(X)$. Since $\Xi$ is a fundamental domain for the action of $\Cr(X,D)$ on $\Eff(X)$,  there exists $g \in \Cr(X,D)$ such that $g(\gamma) \in \Xi$. Moreover, $g(\gamma) \in \cal{M}^{\irr}_X$ by Lemma \ref{lem:Cr(X,D)_preserves_Mirr_internal}.  Therefore, $g(\gamma)$ is one of the finitely many elements in $\cal{M}^{\irr}_X \cap \Xi$ from Step 1, which implies that $\cal{M}^{\irr}_X$ decomposes into finitely many $\Cr(X,D)$-orbits, as claimed.
\end{proof}

\section{Proof of Theorem \ref{thm:main_thm}}

Before proving the main result of this paper, we need the following auxiliary lemma.

\begin{lem}\label{lem:Cr(X,D)_finite_index_in_Cr}
    Let $X$ be a non-generic Sakai surface and let $D\in |{-}K_X|$. Then $\Cr(X,D)$ is a subgroup of $\Cr(X)$ of finite index. 
\end{lem}

\begin{proof}
If $R = A^\smone_0$, then $D$ is irreducible. Since $\Cr(X)$ fixes the class $[-K_X]$, we have $\Cr(X) = \Cr(X,D)$.

In the following, we assume that $R \neq A^\smone_0$, and let $\sigma \in \Cr(X)$. It suffices to show that $\sigma\big(\Comp(D)\big) = \Comp (D)$: indeed, then similarly as in Step 1 of the proof of Lemma \ref{lem_finite_index} we obtain a group homomorphism from $\Cr(X)$ to a finite symmetric group, whose kernel is $\Cr(X,D)$, which proves the lemma.

Assume, for contradiction, that $\sigma\big(\Comp(D)\big) \neq \Comp (D)$. Denote by $\gamma_1,\ldots, \gamma_m$ the classes in $\Delta^{\nod}_{X,D}$ and by $\cal{D}_0,\ldots, \cal{D}_u$ the classes in $\Comp (D)$. Since $\sigma$ preserves the intersection form, the intersection matrix $ \big(\sigma(\gamma_i)\cdot \sigma(\gamma_j)\big)_{i,j=1}^m =  (\gamma_i\cdot \gamma_j)_{i,j=1}^m $ is negative definite by Proposition \ref{prop:nodel_root_system}(b), whereas the intersection matrix 
$$ B:=\big(\sigma(\cal D_s)\cdot \sigma(\cal D_t)\big)_{s,t=0}^u =  (\cal D_s\cdot \cal D_t)_{s,t=0}^u $$
is negative semi-definite
by \cite[\S 4.5]{LSX26b}. This yields that
\begin{equation}\label{eq:234}
\sigma\big(\Comp(D)\big) \neq \Delta^{\nod}_{X,D}.
\end{equation}
On the other hand, since $\sigma$ permutes the classes of $({-}2)$-curves on $X$ by \cite[Proof of Proposition D]{LSX26b}, we have
$$ \Comp(D)\cup \Delta^{\nod}_{X,D} = \sigma\big( \Comp(D)\cup \Delta^{\nod}_{X,D} \big) =   \sigma\big(\Comp(D)\big) \cup \sigma\big(\Delta^{\nod}_{X,D}\big). $$
This, together with \eqref{eq:234} and the assumption that $\sigma\big(\Comp(D)\big) \neq \Comp (D)$, implies that there exists $0 \leq r < s$ such that
$$ \{\sigma(\mathcal{D}_0),\ldots,\sigma(\mathcal{D}_s)\} = \{\mathcal{D}_{i_0},\ldots,\mathcal{D}_{i_r},\gamma_{i_{r+1}},\ldots,\gamma_{i_s}\}. $$
Since $\cal D_i \cdot \gamma_j = 0$ for each  $\cal{D}_i\in \Comp(D)$ and each $ \gamma_j \in \Delta_{X,D}^{\nod}$, the equality above implies that the matrix $B$ is a direct sum of at two matrices. But the matrix $B$ is indecomposable by \cite[\S 4.5 and Appendix B]{LSX26b}, which is a contradiction and finishes the proof.
\end{proof}

We are now ready to prove our main result.

\begin{proof}[Proof of Theorem \ref{thm:main_thm}]
Recall first that $\Nef(X) = \Nef^e(X) = \Nef^+(X)$ and $\Eff(X) = \overline{\Eff}(X)  = \Eff^+(X)$ by \cite[Theorem 3.1]{LSX26a}.

Suppose that there exists a rational polyhedral fundamental domain for the action of $G$ on $\Nef(X)$. In particular, the action of $G$ on $\Nef(X)$ is of polyhedral type, see Definition \ref{dfn:polyhedral_type}(a). Thus, by \cite[Proposition-Definition 4.1]{Loo14}, by duality between $N^1(X)_\R$ and $N_1(X)_\R$ induced by the intersection pairing, and since $G$ preserves the intersection pairing by assumption, the contragradient action of $G$ on the cone $\overline{\Eff}(X)$ is also of polyhedral type. Since $G \subseteq \Cr(X)$ by \cite[Proposition A.1]{LSX26a} and since $\Cr(X)$ preserves $\overline{\Eff}(X)$ by definition, it follows that the action of $\Cr(X)$ on $\overline{\Eff}(X)$ is also of polyhedral type. Therefore, by Lemma \ref{lem:Cr(X,D)_finite_index_in_Cr} and by \cite[Proposition 4.6]{Loo14}, the action of $\Cr(X,D)$ on $\overline{\Eff}(X)$ is of polyhedral type as well. Consequently, the action of $\Cr(X,D)$ on $\Eff^+(X) = \overline{\Eff}(X)$ admits a rational  polyhedral fundamental domain by a result of Looijenga, see \cite[Theorem 2.20]{LSX26a} for instance. But this contradicts Theorem \ref{thm:no_fund_dom_non_generic}, which concludes the proof. 
\end{proof}

\section{Generic nef and effective cones}\label{sec:generic_nef_cone}
Although no variant of the Cone Conjecture holds for non-generic Sakai surfaces,  we are still able to relate the nef cones of generic and non-generic Sakai surfaces of the same type.

\begin{dfn}\label{dfn:eff_gen_nef_gen}
    Let $X$ be a Sakai surface and let $D\in |{-}K_X|$. With the notation from Notation \ref{not:MXNX}, we define the \emph{generic effective cone} $\Eff(X)_{\gen}\subseteq N^1(X)_{\R}$ as the convex hull of $\cal{N}_X\cup \Comp (D)$, and we define the \emph{generic nef cone} as 
    $$ \Nef(X)_{\gen} \coloneqq \{ x \in N^1(X)_{\R} \mid x \cdot \gamma \geq 0 \text{ for each } \gamma \in \cal N_X \cup \Comp(D)\}. $$
\end{dfn}

The term \emph{generic} in the Definition \ref{dfn:eff_gen_nef_gen} will be further explained in Corollary \ref{cor:Mirred}(b)(c). 

The main result of this section is the following, which is similar to \cite[Theorem 3.2]{GHK15} and \cite[Corollary 6.3]{Fri15}.

\begin{thm}\label{thm:nef_fund_domain_for_nefgen}
    Let $X$ be a Sakai surface. Then  $\Nef(X)$ is a fundamental domain for the action of the nodal Weyl group $W^{\nod}$  on the generic nef cone $\Nef(X)_{\gen}$.
\end{thm}

For a Sakai surface $X$ and $D\in|{-}K_X|$, since $\cal{N}_X\subseteq\cal{M}_X\subseteq \Eff(X)$ and $\Eff(X)$ is convex, we have
\begin{equation}\label{eq:eff_gen_inclusion}
    \Eff(X)_{\gen}\subseteq \Eff(X).
\end{equation}
If $X$ is a generic Sakai surface, then $\Delta^{\nod}_{X,D} = \emptyset$ and $W^{\nod}$ is trivial. Moreover, it follows from \cite[Lemma 5.1(d)]{LSX26b} that $\Eff(X)_{\gen}=\Eff(X)$ and $\Nef(X) = \Nef(X)_{\gen}$. Therefore, Theorem \ref{thm:nef_fund_domain_for_nefgen} is trivially satisfied.

As a first step to prove Theorem \ref{thm:nef_fund_domain_for_nefgen}, we introduce an isomorphism between the N\'eron-Severi spaces of different Sakai surfaces of the same type, which allows us to compare their cones of divisors in a common vector space.
 
\begin{lem}\label{lem:isoN1}
	Let $X$ and $X'$ be two Sakai surfaces of the same type $R$, and let $D\in|{-}K_X|$ and $D'\in|{-}K_{X'}|$. Let $\mathcal{E}_0,\ldots,\mathcal{E}_9$ and  $\mathcal{E}'_0,\ldots,\mathcal{E}'_9$ denote the tautological bases of $N^1(X)_{\R}$ and $N^1(X')_{\R}$, respectively, satisfying \eqref{eq:intersectnumbers}. Let $\mathcal{D}_i$ and $\mathcal{D}_i'$ denote the classes of components of $D$ and $D'$, respectively. Then the $\R$-vector space isomorphism $\phi\colon N^1(X)_{\R} \rightarrow N^1(X')_{\R}$ given by $\mathcal{E}_i \mapsto \mathcal{E}'_i$ for all $i=0,\dots,9$ satisfies:
	\begin{enumerate}[\normalfont (a)]
		\item $\phi$ preserves the intersection pairing,
		\item $\phi\big([{-}K_X]\big) = [{-}K_{X'}]$,
		\item $\phi(\cal D_i ) = \cal D_i'$ for every $\cal D_i \in \Comp(D)$,
		\item $\phi({\cal M}_X) = \mathcal{M}_{X'}$.
	\end{enumerate}
The same properties hold for the inverse of $\phi$. 
\end{lem}

\begin{proof}
	Part (a) follows immediately from \eqref{eq:picgroup} and \eqref{eq:intersectnumbers}, whereas (b) follows from \eqref{eq:anticandiv}. Part (c) follows from the fact that $X$ and $X'$ have the same type $R$, and that each $\mathcal{D}_i$ is a linear combination of $\mathcal E_j$. Part (d) follows from (a) and (b) together with \cite[Lemma 2.3(a)]{LSX26b}.
\end{proof}

\begin{rem}\label{rem:transport_group_action}
    The isomorphism $\phi$ from Lemma \ref{lem:isoN1} allows us to transport linear group actions on different N\'eron--Severi spaces via conjugation to the same real vector space. In particular, for any Sakai surface $X$ with $D\in |{-}K_X|$, the group $\Cr(X,D)$ can be considered as a subgroup of  $\GL(N^1(X')_{\R})$ via  the injective homomorphism 
$$ \Cr(X,D)  \rightarrow \GL(N^1(X')_{\R}),\quad g  \mapsto g' \coloneqq \phi\circ g \circ \phi^{-1}. $$
\end{rem}

Recall that for every Sakai surface $X$ of type $R$, there exists a generic Sakai surface $X_{\gen}$ of the same type by \cite[Proposition B.5]{LSX26a}. Let $D_{\gen} \in |{-}K_{X_{\gen}}|$. In the sequel, we will frequently consider the isomorphism
\begin{equation}\label{eq:phi_gen}
    \phi \colon N^1(X_{\gen})_{\R} \to N^1(X)_{\R}
\end{equation}
defined in Lemma \ref{lem:isoN1}, which identifies $\Cr(X_{\gen}, D_{\gen})$ as a subgroup of $\GL\big( N^1(X)_{\R}\big)$.

Using the isomorphism \eqref{eq:phi_gen}, we can identify the generic nef cone and the generic effective cone of a Sakai surface with the nef cone and the effective cone of a generic Sakai surface, respectively.

\begin{cor}\label{cor:Mirred}
	As in Notation \ref{not:MXNX}, let $X$  be a Sakai surface and consider the isomorphism $\phi\colon N^1(X_{\gen})_{\R} \rightarrow N^1(X)_{\R}$ as in \eqref{eq:phi_gen}. Then:
    \begin{enumerate}[\normalfont (a)]
        \item $ \phi(\mathcal{M}^{\irr}_{X_{\gen}}) = \mathcal{N}_X, $ in particular, the set $\phi(\mathcal{M}^{\irr}_{X_{\gen}})\subseteq N^1(X)_{\R}$ does not depend on the choice of $X_{\gen}$,
        \item $\Eff(X)_{\gen} = \phi\big(\Eff(X_{\gen})\big) \subseteq N^1(X)_{\R}$ is a closed cone,
      \item $\Nef(X)_{\gen} = \phi\big(\Nef(X_{\gen})\big) \subseteq N^1(X)_{\R}$ is dual to the cone $\Eff(X)_{\gen}$.
    \end{enumerate}
\end{cor}

\begin{proof}
Part (b) follows directly from part (a) and \cite[Proposition 5.1(a)(d)]{LSX26b}. Part (c) follows from part (b) and Lemma \ref{lem:isoN1}(a).

Now we show (a). Let $D\in|{-}K_X|$ and $D_{\gen}\in|{-}K_{X_{\gen}}|$. In $N^1(X)_{\R}$ we have, by \cite[Lemma 2.6]{LSX26b},
	$$ \mathcal{M}^{\irr}_X = \big\{ E \in \mathcal M_X \mid E\cdot C \geq 0 \text{ for all } C \in \Comp(D) \cup \Delta^{\nod}_X \big\}. $$
	Similarly, by Definition \ref{dfn:goodHalphen} and by Lemma \ref{lem:isoN1} we have
	\begin{align*}
	\phi&(\mathcal{M}^{\irr}_{X_{\gen}}) = \big\{ \phi(E)\mid E\in \mathcal M_{X_{\gen}}, \ E\cdot \cal{D}_{\gen} \geq 0 \text{ for all } \cal{D}_{\gen} \in \Comp(D_{\gen})\big\}\\
	&= \big\{ \phi(E)\mid E\in \mathcal M_{X_{\gen}}, \ \phi(E)\cdot \phi(\cal{D}_{\gen}) \geq 0 \text{ for all } \cal{D}_{\gen} \in \Comp(D_{\gen}) \big\}\\
	& = \big\{ F\in \mathcal M_X \mid F\cdot \cal{D} \geq 0 \text{ for all } \cal{D} \in \Comp(D) \big\}.	
	\end{align*}
	This shows the result.
\end{proof}

\begin{lem}\label{lem:cr_gen}
    Let $X$  be a Sakai surface and let $\phi\colon N^1(X_{\gen})_{\R} \rightarrow N^1(X)_{\R}$ be the isomorphism as in \eqref{eq:phi_gen}. Then the induced action of $\Cr(X_{\gen},D_{\gen})$ on $N^1(X)_{\R}$ preserves $\Nef(X)_{\gen}$ and $\Eff(X)_{\gen}$
\end{lem}

\begin{proof}
     Let $g \in \Cr(X_{\gen},D_{\gen})$ and set $g' \coloneqq \phi \circ g \circ \phi^{-1} \in \GL(N^1(X)_{\R})$ as in Remark \ref{rem:transport_group_action}. Then by Corollary \ref{cor:Mirred}(b)(c) we have
$$ g'\big(\Nef(X)_{\gen}\big) = \phi\big(g(\Nef(X_{\gen})\big) = \phi\big(\Nef(X_{\gen})\big) = \Nef(X)_{\gen}, $$
and similarly, $g'\big(\Eff(X)_{\gen}\big)=\Eff(X)_{\gen}$. 
\end{proof}

The next lemma shows some properties of the generic nef cone of a Sakai surface.
   \begin{lem}\label{lem:nefgen_in_cplus}
   	Let $X$ be a Sakai surface. Then: 
   	\begin{enumerate}[\normalfont (a)]
    \item $\Nef(X)  \subseteq \Nef(X)_{\gen}$,
   	\item $ \Nef(X)_{\gen} \subseteq \overline{\mathcal{C}^+_X}$,
    \item $ \Nef(X)_{\gen} = \NefE(X)_{\gen}$,
   	\item $ \Nef(X)_{\gen} \subseteq \mathcal{T}(E^{(1)}_8) \subseteq \mathcal{T}(R^\smperp).$
   	\end{enumerate}
   	\end{lem}
   	
   \begin{proof} 
   We first note that part (a) follows immediately from \eqref{eq:eff_gen_inclusion} and Corollary \ref{cor:Mirred}(c).
   
   Let $X_{\gen}$ be a generic Sakai surface of the same type as $X$, and 
let $\phi\colon N^1(X_{\gen})_{\R} \to N^1(X)_{\R}$ be the isomorphism as in \eqref{eq:phi_gen}.

For part (b), let $x \in \Nef(X)_{\gen}$. Then $y \coloneqq \phi^{-1}(x) \in \Nef(X_{\gen})$, and hence $y^2 \geq 0$. Therefore, Lemma \ref{lem:isoN1}(a) yields
   	$$ x^2  = \big(\phi(y)\big)^2 = y^2 \geq 0, $$
   	which shows (b).
  
    For part (c), it suffices to show that $\Nef(X)_{\gen} \subseteq \Eff(X)$.
We first note that $\Nef(X_{\gen}) \subseteq \Eff(X_{\gen})$ by \cite[Theorem 3.1(b)]{LSX26a}. Thus, using the isomorphism $\phi$ and \eqref{eq:eff_gen_inclusion}, we obtain 
	$$ \Nef(X)_{\gen} \subseteq \Eff(X)_{\gen} \subseteq \Eff(X). $$

Finally, we show part (d). Let $x \in \Nef(X)_{\gen}$. By part (c) we have $\delta\cdot x \geq 0$.  If  $\delta\cdot x > 0$, then $x \in  \mathcal{T}(E^{(1)}_8)$ by \eqref{eq:tautological_tits_cone}. Now assume that $\delta \cdot x = 0$. Moreover, as $x^2 \geq 0$ by part (b), we infer from \cite[Lemma 2.2]{LH05} that $x \in \Z \delta$. Again using \eqref{eq:tautological_tits_cone}, we obtain $x \in \mathcal{T}(E^{(1)}_8)$. The second inclusion follows from \cite[Proposition 4.1(c)]{LSX26b}.
   	\end{proof}
In the remainder of this section, we restrict to non-generic Sakai surfaces. By Corollary \ref{cor:generic_types}, this implies that the type $R$ satisfies  
 \begin{equation}\label{eq:non_generic_types}
   R \notin \{ A_7^\smone, A_8^\smone, D_7^\smone, D_8^\smone, E_8^\smone \}.   
 \end{equation}
    Using the negative fundamental chamber defined by $\Delta^{\nod}_{X,D}$ as in Section \ref{sec:nodal_weyl}, we can refine Lemma \ref{lem:nefgen_in_cplus}(a).
    
    \begin{lem}\label{lem:nef_fundchamber}
    Let $X$ be a non-generic Sakai surface. Then 
    	$$\Nef(X) = \Nef(X)_{\gen} \cap \overline{\cal{F}^{\nod}}. $$
    \end{lem}
    
    \begin{proof}
    	One inclusion is clear by Lemma \ref{lem:nefgen_in_cplus}(a). For the opposite inclusion, assume that $x \in  \Nef(X)_{\gen}$ satisfies $ x\cdot \gamma \geq 0$ for all $\gamma\in \Delta_{X,D}^{\nod}$. Then $x  \cdot \gamma \geq 0$ for all $\gamma \in \mathcal{N}_X\cup \Comp(D) \cup \Delta_{X,D}^{\nod}$. Since $\mathcal{M}^{\irr}_X \subseteq \mathcal{N}_X$ by \cite[Lemma 2.6]{LSX26b}, we conclude that $x \cdot \gamma \geq 0$ for all $\gamma \in \mathcal{M}^{\irr}_X\cup \Comp(D) \cup \Delta_{X,D}^{\nod}$. Hence, by \cite[Lemma 5.1(c)]{LSX26b} we obtain $x\in \Nef(X)$, as desired.
    \end{proof}
    
\begin{lem}\label{lem:wnod_preserves_nef_gen}
	Let $X$ be a non-generic Sakai surface of type $R$, and let $X_{\gen}$ be a generic Sakai surface of the same type. Let $D_{\gen} \in |{-}K_{X_{\gen}}|$, and consider $\Cr(X_{\gen},D_{\gen})$ as a subgroup of $\GL \big(N^1(X)_{\R}\big)$ via the isomorphism \eqref{eq:phi_gen}. Then $W^{\nod}$ is a subgroup of $\Cr(X_{\gen},D_{\gen})$. In particular, $W^{\nod}$ preserves $\Nef(X)_{\gen}$.
	
\end{lem}
\begin{proof} 
In the following, we show that $W^{\nod}$ is contained in $\Cr(X_{\gen},D_{\gen})$. This, together with Lemma \ref{lem:cr_gen}, implies that $W^{\nod}$ preserves $\Nef(X)_{\gen}$.

If $R \neq A^{\smone}_6$, then together with \eqref{eq:non_generic_types}, we can apply \cite[Theorem 26(a)]{Sak01} to obtain $\Cr(X_{\gen},D_{\gen}) = W(R^\smperp)$. Thus by Lemma \ref{lem:nodal_roots}(a), we obtain $W^{\nod} \subseteq W(R^\perp)=\Cr(X_{\gen},D_{\gen})$, as desired.

If $R = A^{\smone}_6$, then by \eqref{eq:decompose_W} and Remark \ref{rem:the_case_A6} we have $W(R^\smperp) = W_{\mathscr{D}} W_{\mathscr{D}'}$ 
and $W^{\nod} \subseteq W_{\mathscr{D}'}$. Moreover, by \cite[Setup I(b)]{LSX26b} $W_{\mathscr{D}'}$ is a subgroup of $\Cr(X_{\gen},D_{\gen})$.
Therefore, $W^{\nod}$ is a subgroup of $\Cr(X_{\gen},D_{\gen})$, which finishes the proof. 
\end{proof}

Now we can show our main result of this section.

\begin{proof}[Proof of Theorem \ref{thm:nef_fund_domain_for_nefgen}]
We may assume that  $X$ is non-generic. First, we infer from Lemmas \ref{lem:nef_fundchamber} and \ref{lem:wnod_preserves_nef_gen} that 
$$ W^{\nod} \cdot \Nef(X) \subseteq \Nef(X)_{\gen}. $$
To show the equality, let $x \in \Nef(X)_{\gen}$. Then $x \in \cal{T}({R^{\nod}})$ by Lemmas \ref{lem:nefgen_in_cplus}(d) and \ref{lem:nodal_roots}(c). Thus, there exists an element $w\in W^{\nod}$ such that $w(x)\in \overline{\mathcal{F}^{\nod}}$. Since $W^{\nod}$ preserves $\Nef(X)_{\gen}$ by Lemma \ref{lem:wnod_preserves_nef_gen}, we conclude that
	$$ w(x) \in \overline{\cal F^{\nod}} \cap \Nef(X)_{\gen} = \Nef(X). $$ 
Finally, let $w\in W^{\nod}$ such that $\Int\big( w\big(\Nef(X)\big) \big)\cap \Int \big(\Nef(X) \big) \neq \emptyset $. Since
    $$ \Int\big( w\big(\Nef(X) \big) \big)\cap \Int \big(\Nef(X) \big) \subseteq \Int\big( w( \overline{\cal F^{\nod}} ) \big)\cap \Int \big( \overline{\cal F^{\nod}}  \big) $$
    by Lemma \ref{lem:nef_fundchamber}, we obtain $\Int\big( w \big( \overline{\cal F^{\nod}}\big) \big)\cap \Int \big( \overline{\cal F^{\nod}} \big) \neq \emptyset $. By \cite[Proposition 3.12]{LSX26b}, this implies that $w$ is the identity, which finishes the proof.
\end{proof}

\begin{rem}\label{rem:W_nod_not_normal}
    Using the notion of generic nef cones, the result of \cite[Theorem C]{LSX26b} translates to the following. Let $X$ be a Sakai surface. Then the action of $\Cr(X_{\gen},D_{\gen})$ on $\Nef(X)_{\gen}$ admits a rational polyhedral fundamental domain. In view of this result, together with Lemma \ref{lem:wnod_preserves_nef_gen} and Theorem \ref{thm:nef_fund_domain_for_nefgen}, one might expect that the quotient $\Cr(X_{\gen},D_{\gen})/W^{\nod}$ acts on $\Nef(X)$ with a rational polyhedral fundamental domain; see also \cite[Remark 9.2]{Fri15}. We show, however, that this approach does not work because $W^{\nod}$ is never a normal subgroup of $\Cr(X_{\gen},D_{\gen})$.
\end{rem}

\begin{lem}\label{lem:not_normal}
Let $X$ be a non-generic Sakai surface of type $R$, and let $X_{\gen}$ be a generic Sakai surface of the same type. Let $D_{\gen} \in |{-}K_{X_{\gen}}|$. Then $W^{\nod}$ is not a normal subgroup of $\Cr(X_{\gen},D_{\gen})$.
\end{lem}

\begin{proof}
Suppose that $W^{\nod}\subseteq \Cr(X_{\gen}, D_{\gen})$ is a normal subgroup. Recall from Proposition \ref{prop:nodel_root_system}(c) that $W^{\nod}$ is finite. Moreover, it follows from \eqref{eq:non_generic_types} and \cite[Theorem 26(a)]{Sak01} that
$$\Cr(X_{\gen}, D_{\gen})= W(R^{\smperp}) \text{ if } R\neq A_6^{\smone}.$$

If $R^{\smperp}$ is indecomposable, then $\Cr(X_{\gen}, D_{\gen})$ is a Weyl group of affine type. But then  \cite[Corollary 0.3]{Max98} implies that the finite group $W^{\nod}$ is of finite index in the infinite group $\Cr(X_{\gen}, D_{\gen})$, a contradiction. 

Now if $R^{\smperp}$ is decomposable, that means \ $R  \in \{A^{\smone}_5,A^{\smone}_6,D^{\smone}_6\}$ by \cite[\S 4.8]{LSX26b}, then 
$$W(R^\smperp) = W_{\mathscr{D}}W_{\mathscr{D}'}$$ 
by \eqref{eq:decompose_W}.
If $R = A^{\smone}_6$, then by Remark \ref{rem:the_case_A6} and \cite[Setup I(c)]{LSX26b} we have 
$$W^{\nod} \subseteq W_{\mathscr{D}'} \subseteq \Cr(X_{\gen},D_{\gen}).$$ 
Thus $W^{\nod}$ is a finite normal subgroup of the affine Weyl group $W_{\mathscr{D}'}$, which again contradicts \cite[Corollary 0.3]{Max98}. 

If $R\in \{A^{\smone}_5,D^{\smone}_6\}$, let $ \gamma \in \Delta^{\nod}_{X,D}$. Then by \cite[Proposition 4.1(d)]{LSX26b}, we infer that $\gamma$ is either in $\Delta_{\mathscr{D}}^{\re}$ or in $\Delta_{\mathscr{D}'}^{\re}$. We prove the case when $\gamma \in \Delta_{\mathscr{D}}^{\re}$, the other case is analogous. By assumption, $W^{\nod}$ is a normal subgroup of $W_{\mathscr{D}}W_{\mathscr{D}'}$. Thus, the non-trivial finite group $W^{\nod} \cap W_{\mathscr{D}}$ is a normal subgroup of the affine Weyl group $W_{\mathscr{D}}$, which again contradicts \cite[Corollary 0.3]{Max98}.
\end{proof}

\bibliographystyle{amsalpha}
\bibliography{biblio}

\end{document}